\documentclass[12pt]{amsart}
\usepackage{amsmath, amsthm, amscd, amsfonts,amssymb}

\newtheorem{theorem}{Theorem}[section]
\newtheorem{lemma}[theorem]{Lemma}

\newtheorem{corollary}[theorem]{Corollary}

\theoremstyle{definition}

\theoremstyle{remark}

\numberwithin{equation}{section}

\newfont{\kh}{msbm10}

\begin{document}
\title[ Automorphism group of power graphs ]
{A description of Automorphism group of power graphs of  finite groups  }
\author{S. H. Jafari}
\address{S. H. Jafari, \newline Faculty of Mathematics,
Shahrood University of Technology, P. O. Box 3619995161-316,
Shahrood, Iran} \email{shjafari55@gmail.com }

\begin{abstract}
The power graph of a group is the graph whose vertex set is the set
of nontrivial  elements of group, two elements being adjacent if one is a power
of the other. We introduce some way for find  the automorphism groups of
 some graphs. As an application We describe the full automorphism group of the power graph
of all finite groups.
 Also we obtain the full automorphism group of  power graph of abelian, homocyclic  and nilpotent groups.
\end{abstract}

\maketitle
\section{Introduction.}

 The directed power graph of a semigroup $S$ was defined by Kelarev and Quinn \cite{kq} as the
digraph $\mathcal{P}(S)$ with vertex set S, in which there is an arc from
$x$ to $y$ if and only if $x\neq y$ and $y$$=$$x^m$ for some positive
integer $m$. Motivated by this, Chakrabarty et al. \cite{css}
defined the (undirected) power graph $\mathcal{P}(S)$, in which distinct
$x$ and $y$ are joined if one is a power of the other.
The concept of power graphs  has been studied extensively
by many authors. For a list of references and the history of this topic, the reader is
referred to [2, 5-10].\\

 Let $L$ be a graph. We denote $V(L)$  and $E(L)$ for
vertices and edges of $L$, respectively. We use $a-b$ if $a$ is adjacent
  to $b$. Also for a subgraph $H$ of
$L$ and $a$$\in$$ V(H)$, we denote $H-a$ for the subgraph
generated by $V(H)-\{a\}$.  The (open) neighborhood $N(a)$ of
vertex $a$$\in$$ V(L)$ is the set of vertices are adjacent to $a$. Also
the closed neighborhood of $a$, $N[a]$ is $N(a)\cup \{a\}$.\\
Throughout this paper, all groups and graphs  are  finite
 and the following
notation is used: $Aut(G)$ denotes the group of automorphisms of $G$;
 $\mathbb{Z}_m$ the cyclic group of order $m$; $\mathbb{Z}_m^n$
  the direct product of $n$ copies
of $\mathbb{Z}_m$. \\

In this paper we describe the automorphism group of the power graph
of finite group. Also we obtain automorphism group of the power graph of abelian, and  homocyclic   groups.




\section{automorphism group of graphs }
 In this section we provide some ways for calculating automorphism groups of graphs.
 Let $L_1,L_2$ be two  graphs, a function $f :(V(L_1)\cup E(L_1))\longrightarrow (V(L_2) \cup E(L_2))$ to $V(L)$  is said an isomorphism if $f$ is bijective, $f(V(L_1))$$=$$f(V(L_2)), f(E(L_1))$$=$$f(E(L_2))$  and,  $x-y $$\in$$ E(L_1)$ if and only if  $f(x)-f(y) $$\in$$ E(L_2)$.\\
  We say a subset $H$ of $V(L)$ is an  $MEN$-subset if it is
maximal subset which any two elements of $H$ have equal closed
neighborhood in $L$. We denote $H$ by $\overline{a}$ for any $a $$\in$$ H$.
  We define the weighted graph $\overline{L}$  as follows.\\
   Let $V(\overline{L})$$=$$\{\overline{x}|x$$\in$$ V(L) \}$, $weight(\overline{x})$$=$$|\overline{x}|$, and two vertices $\overline{x}, \overline{y}$  are adjacent if $x$ and $y$ are adjacent in $L$. Also  in weighted graph any automorphism preserves the weight of each element.
  \begin{theorem}[ \cite{j2} {Theorem 2.2}]\label{t1}
For a graph $L$ with $|V(L)|<\infty$,
\begin{center}
{$Aut(L)$$\cong$$ Aut(\overline{L}) \ltimes \prod_{B \in V(\overline{L})}S_{|B|}$.}
\end{center}
\end{theorem}

\begin{theorem}\label{t2}
Let $L$$=$$L_1 \cup L_2 \cup \cdots \cup L_t$  be a finite graph and $L_1 $$\cong$$ L_2 $$\cong$$\cdots $$\cong$$ L_t$. If\\ $\varphi(L_i)$$\in$$ \{L_1, \cdots,L_t\}$ for all $i$$\in$$ \{1, \cdots, t\}$ and $ \varphi $$\in$$ Aut(L)$, then $Aut(L)$$\cong$$ Aut(L_1)\wr S_t$.
\end{theorem}
\begin{proof}
We have $Aut(L)$ act on $A$$=$$\{ V(L_1), \cdots, V(L_t)\}$ and then there exist group homomorphism $\psi: Aut(L)\longrightarrow S_A $  such that $ker(\psi)$$=$$ \{ \varphi|\varphi(L_i)$$=$$L_i$ for all $ i\}$.  Consequently
 \begin{center}
 {$ker(\psi)$$\cong$$ Aut(L_1)\times \cdots \times Aut(L_t)$.
 }
 \end{center}
  Since $L$ is finite, so there exists a  totally  partial order $\leq$  on $V(L_1)$. Assume that $L_1$$\cong$$_{f_i}L_i$ for $i\geq 2$ and $f_1$$=$$Id_{L_1}$. We consider  $f_i(u) \leq f_i(v)$ if $u\leq v$ for all $u,v $$\in$$ V(L_1)$. Thus we give a totally partial order on each $V(L_i)$.
Let $H$$=$$\{ \varphi \in Aut(L)| u\leq v$  if and only if  $\varphi(u) \leq \varphi(v)\}$. We see that if $\varphi $$\in$$ H$ and $\varphi(L_i)$$=$$L_i$ then $\varphi$ is identity on $L_i$. Let $x$ be the minimum element of $V(L_1)$ and
  $B$$=$$\{ f_i(x)| 1\leq i\leq t\}$. Thus each element of $H$ is induced an bijection function  on $B$. Also for any bijection function $\sigma$ on $B$, the function $\varphi$ by definition $\varphi(u)$$=$$f_j((f_i)^{-1}(u))$,  whence  $u $$\in$$ V(L_i)$ and $\sigma(f_i(x))$$=$$f_j(x)$ is an automorphism of $L$. But $\ker(\psi) \cap H$ is trivial,
  consequently  $|Aut(L)|\geq |H||ker(\psi)|$. On the other hand $|Aut(L)/ker(\psi)| \leq |S_t|$, which completes the proof.
\end{proof}


\section{the automorphism group of power graph of finite groups}
In this section we  describe the automorphism group of finite groups, directed product of some groups and nilpotent groups.\\
By using Theorem \ref{t1} we have the following which is same to main result of M. Feng, X. Ma and K. Wang in 2016.
\begin{theorem}\label{autgroup}
For a finite group  $G$,
\begin{center}
{$Aut(\mathcal{P}(G))$$\cong$$ Aut(\overline{\mathcal{P}(G)}) \ltimes \prod_{\overline{x} \in V(\overline{\mathcal{P}(G)})}S_{|\overline{x}|}$.}
\end{center}
\end{theorem}
For the cyclic subgroup $\langle a\rangle$ of  the group $G$, the
subset $\{a^i|(i,o(a))$$=$$1\}$ of  $G$,
is denoted by $gen(\langle a\rangle)$.

We will use the following.
\begin{lemma} [\cite{j1}, {Proposition 2.8}] \label{l2}
Let $G$ be a finite group and $a $$\in$$ G$. If $|C_G(a)|$ is not prime power then
$gen(\langle a\rangle)$ is an MEN-subset.

\end{lemma}
\begin{lemma} [\cite{j1},{Proposition 2.9}] \label{l3}
 Let $\langle a\rangle$ be a maximal cyclic subgroup of finite group $G$. If
$ C_G(a)\neq \langle a\rangle$ then $gen(\langle b \rangle)$ is an MEN-subset
for any $b $$\in$$ \langle a\rangle$.

\end{lemma}
\begin{lemma} [\cite{j1},{Theorem 2.10}]\label{hal3}
Let $G$ be a finite group. Then $K\subseteq G$ is an
$MEN$-subset if and only if $K$ satisfies in one of the following
conditions:

(1) $K=\langle a\rangle - \langle a^{p^t} \rangle$ where
$o(a)=p^n$, $1<t \leq n$, $N[a^{p^{t-1}}]= \langle a\rangle$ and
$N[a^{p^t}]\neq \langle a\rangle $.

 (2) $K=gen(\langle a\rangle)$ for some $a \in G$.
 \end{lemma}
 Let $x_M$ is an element of maximum order in $\overline{x}$. By Lemmas \ref{l2}, \ref{l3} and \ref{hal3}, we have the following.

\begin{corollary}\label{ha1}
    Let $G$ be a finite group and $x\in G$. Then $o(x_M)=1+\sum_{\overline{a} \in N(\overline{x_M}), |\overline{a}| \leq |\overline{x_M}| } |\overline{a}|$.
   \end{corollary}
   \begin{theorem}
   Let $G=H\times K$ and $(|H|, |K|)=1$ then
   \begin{center}{$Aut(\overline{\mathcal{P}(G)})=Aut(\overline{\mathcal{P}(H)}) \times Aut(\overline{\mathcal{P}(K)})$.
   }
   \end{center}

   \end{theorem}
   \begin{proof}
       Let $H,K$ are nontrivial groups, $(a,b) $$\in$$ G $ and $\overline{\varphi} $$\in$$ Aut(\overline{\mathcal{P}(G)})$. Then $C_G(a,b)$$=$$C_H(a) \times C_K(b)$ and, $|C_G(a,b)|$ is not  a prime power.
       Thus by Lemma \ref{l2} and  Corollary \ref{ha1}, $\varphi$ is preserving the order of each element of $G$.
       Therefore $\varphi(H)$$=$$H, \varphi(K)$$=$$K$. Assume that $\varphi(a)$$=$$a_1, \varphi(b)$$=$$b_1$, we have
       $o(\varphi(a,b))$$=$$o(a,b)$ and $a_1, b_1 $$\in$$ \langle \varphi(a,b) \rangle$.   But exactly subgroup of order $o(a)o(b)$ containing
       $a_1$ and $b_1$ is $\langle (a_1, b_1)\rangle $. We deduce
       that $\overline{\varphi(a,b)}$$=$$\overline{(a_1,b_1)}$.
       Therefore\\ $Aut(\overline{\mathcal{P}(G)})$$\cong$$ Aut(\widetilde{\mathcal{P}(H)}) \times Aut(\widetilde{\mathcal{P}(K)})$
       where $\widetilde{\mathcal{P}(H)}$ is the set of       $MEN$-set of $\mathcal{P}(G)$ contained
       in $H$. But each  element   of   $\overline{\mathcal{P}(H)}$ is same to an element of $ \widetilde{\mathcal{P}(H)}$,  or  union
       of some elements of $ \widetilde{\mathcal{P}(H)}$. By Lemma \ref{hal3}, we can assume
       that $\overline{x_H}=\overline{x_1}\cup \cdots \cup \overline{x_t}$ where $o(x_1)< \cdots <o(x_t)=o(x)$ and $N_H[x_1]= \cdots=N_H[x_t]=N_H[x]$. Since we can consider these points as one point in $\widetilde{\mathcal{P}(H)}$,  hence the
       desired result follows.

   \end{proof}
   A direct result  of  above theorem is for nilpotent groups as following.

          \begin{theorem}\label{nil}
          Assume that  $G$  be a nilpotent  finite group and,  $G=P_1 \times \cdots \times P_t$ where $P_i$ is sylow subgroup of $G$.
          Then
           \begin{center}{
           $Aut(\mathcal{P}(G))$$=$$(Aut(\overline{\mathcal{P}(P_1)})\times \cdots \times Aut(\overline{\mathcal{P}(P_t)}))\ltimes \prod_{B \in V(\overline{\mathcal{P}(G)})}S_{|B|}$.}
           \end{center}

          \end{theorem}

          Now we find the automorphism group of power graph of
           cyclic group when $n$ is not prime power, which is same to \cite{feng}.
   \begin{corollary}
    Let $G$ be a  cyclic group of order $n$ where $n$ is not prime power.
     Then $Aut(\mathcal{P}(G))$$\cong$$  \prod_{d|n, d>1}S_{\Phi(d)}$, whence
    $\phi$ is Euler-function.
  \end{corollary}
  \begin{proof}

    Since $|C(x)|$ is not prime power for all $x$$\in$$ G$, by Lemma \ref{l3},
      $|\overline{a}|$$=$$ \phi(o(a))$. So
      \begin{center}{
      $Aut(\mathcal{P}(G))$$=$$(Aut(\overline{\mathcal{P}(P_1)})\times \cdots \times Aut(\overline{\mathcal{P}(P_t)}))\ltimes \prod_{d|n, d>1}S_{\Phi(d)}$.}
      \end{center}
      But any sylow subgroup of $G$ is cyclic, and $|\overline{P_i}|=1$, as desired.
  \end{proof}

\section{abelian groups}

In this section we  certainly  calculate the automorphism group of  power graph of homocyclic and abelian  finite groups.\\

 Let $G$ be a finite abelian $p$-group and $x$ a nontrivial element of $G$.
 The height of $x$, denoted by $height(x)$, is the largest power
  $p^n$ of the prime $p$ such that $x $$\in$$ G^{p^n}$.
  A non-cyclic  group $G$ said a homocyclic group if $G$ be  a directed product of some
   copes of cyclic group of order $p^m$  for some integer $m$.   \\
We begin by a famous theorem in group theory which is played main rule in this section.
\begin{theorem} \label{d1}
Let $G$ be a finite abelian  group and $a$ be an element of  $G$ where $o(a)$$=$$exp(G)$.
Then there exist a subgroup $H$ of  $G$ such that $G$$=$$\langle$$ a$$\rangle$$ \times H$
\end{theorem}

\begin{lemma}\label{ta1}
Let $G$ be a homocyclic group. Then $Aut(G)$, the automorphism group of $G$, acts
transitively on the set of elements with  equal orders.
\end{lemma}
\begin{proof}
Let $G $$\cong$$  \mathbb{Z}_{p^m}^n$ and $a,b$  be two  elements of order $p^t$.
Since $G$ is homocyclic, \\ $height(a)$$=$$height(b)$$=$$p^{m-t}$. So there exist $x,y$$\in$$ G$ such that $x^{p^{m-t}}$$=$$a$
 and $y^{p^{m-t}}$$=$$b$.
By Theorem \ref{d1}, there exist subgroups $H_1, H_2$ such that
 $G$$=$$\langle x \rangle \times H_1$$=$$\langle y \rangle \times H_2$.
  From which   $H_1$$\cong$$ H_2 $$\cong$$ \mathbb{Z}_{p^m}^{n-1}$.
  Assume that $H_1$$\cong$$_{\varphi} H_2$. Now $\psi$ by definition
   $\psi(x^ih)$$=$$y^i\varphi(h)$ where $h$$\in$$ H_1$ and $0\leq i\leq p^m$, is an automorphism of $G$ and
    $\psi(a)$$=$$b$, as required.
\end{proof}

\begin{theorem}
For $G$$\cong$$ \mathbb{Z}_{p^m}^n$,
\begin{center}{
$Aut(\mathcal{P}(G))$$=$$(( \cdots(S_{k_m}  \wr  \cdots)\wr S_{k_2})\wr S_{k_1})
\ltimes (\prod_{i=1}^m S_{(p^i-p^{i-1})}^{r_i})$,}
\end{center}
 where $r_t$$=$$(p^{tn}-p^{(t-1)n})/(p^t-p^{t-1})$,
$k_1$$=$$r_1$ and $k_{i+1}$$=$$r_{i+1}/r_i$.
\end{theorem}
\begin{proof}
By Theorem \ref{t1}, $Aut(\mathcal{P}(G))$$\cong$$ Aut(\overline{\mathcal{P}(G)}) \ltimes \prod_{B
 \in V(\overline{\mathcal{P}(G)})}S_{|B|}$.
   Since $G$ is  non-cyclic abelian  group, by Lemma \ref{l2}, $|\overline{a}|$$=$$p^t-p^{t-1}$, where
    $o(a)$$=$$p^t$.\\
  Set $R_t$$=$$\{\overline{x}| o(x)$$=$$p^t\}$ and $r_t$$=$$|R_t|$. We know that $G$ has exactly
   $p^{tn}-p^{(t-1)n}$ elements of order $p^t$,  thus
        $r_t$$=$$(p^{tn}-p^{(t-1)n})/(p^t-p^{t-1})$.  From
     which  the second part of semi-directed product of theorem has been found.\\
     Now we  want to find the first part of that product.

In  a $p$-group, two elements $a,b$  are in one  connected components
 of $\mathcal{P}(G)$ if and only if $\langle a \rangle \cap \langle b \rangle \neq \{1\}$.
 So $\mathcal{P}(G)$ has exactly $r_1$ connected components. On the other hand  by Lemma \ref{ta1},
 $Aut(G)$, and so $Aut(\mathcal{P}(G))$ acts transitively on the set of elements of order $p$.
 Thus   $Aut(\mathcal{P}(G))$ acts transitively   on $R_1$. Consequently all connected components
 of $\overline{\mathcal{P}(G)}$ are isomorphic. By Theorem \ref{t2}, $Aut(\overline{\mathcal{P}(G)})$$=$$Aut(K_1)
 \wr S_{r_1}$ where $K_1$ be a one of connected components of $\overline{\mathcal{P}(G)}$.
 But there is only one element, say  $\overline{a_1}$, in  $V(K_1)$ with properties $|\overline{a}|$$=$$p-1$
 and  $N[\overline{a_1}]$$=$$V(K_1)$, from which $Aut(K_1)$$=$$Aut(K_1-\overline{a_1})$.
    Now two elements $\overline{a},\overline{b}$  are in one  connected components
 of $K_1-\overline{a_1}$ if and only if $\langle a \rangle \cap \langle b \rangle \neq
 \langle a_1 \rangle$.
Since all connected  components of  $\overline{\mathcal{P}(G)}$ are isomorphic and $Aut(\mathcal{P}(G))$ acts
transitively on $R_2$, then  $K_1-\overline{a_1}$ has exactly
$k_2$$=$$r_2/r_1$  isomorphic connected  components. It follows that $ Aut(K_1)$$\cong$$ Aut(K_2)\wr S_{k_2}$  whence $ K_2$ be a connected  component of $K_1-\overline{a_1}$.\\
By following  this process the proof is completed.

\end{proof}

           Let $G$ be a finite $p$-group and $exp(G)$$=$$p^n$. Set $\Omega_t(G)$$=$$\{x|x^{p^t}$$=$$1\}$ and,
 \begin{center}{
 $H_{t}(G)$$=$$\{x $$\in$$ G| o(x)$$=$$p, height(x)$$=$$p^{t-1}\}$.}
 \end{center}

   \begin{lemma}\label{h111}
    Let $G$ be an abelian $p$-group and $x$ is an element of order $p$. Then there is an element $a$  and subgroup $L$ of $G$ such that
    $ G$$=$$\langle a \rangle$$\times$$L$ and $x\in \langle a \rangle $.
   \end{lemma}
      \begin{proof}
           Since $G$ is abelian,  then  $G $$\cong$$  G_1 \times \ldots \times G_k$
           where  $G_1,  \ldots, G_k$ are non-isomorphic homocyclic groups. Assume that $exp(G_i)$$=$$p^{n_i}$ and
           $p^{n_1}<  \cdots< p^{n_t}$. Then
           $x$$=$$(x_1, \ldots, x_k) $$\in$$ G$ has order $p$ if and only if $max\{o(x_i)| i$$\in$$\{1, \ldots, k\} \}$$=$$p$.
           Also $height(x)$$\in$$ \{p^{n_1-1},  \ldots, p^{n_t-1}\}$ and,
           $x$$\in$$ H_{n_t}(G)$ if and only if $x_1$$=$$x_2$$=$$ \cdots$$=$$x_{t-1}$$=$$1$ and $o(x_t)$$=$$p$.
           Assume that  $a$$=$$(a_1,  \ldots, a_k)$$\in$$ G$  and $a^{p^{n_t-1}}$$=$$x$.\\
           Therefore $o(a_t)$$=$$p^{n_t}$. By Theorem \ref{d1}, there is a subgroup $K$
           such that
           $G_t$$=$$\langle a_t \rangle \times K$ and then      $G$$=$$\langle a_t \rangle \times  G_1 \times  \cdots \times G_{t-1} \times K \times G_{t+1}\times  \cdots \times G_k$.
           So there is a subgroup $L$ with $G$$=$$\langle a_t \rangle \times L$. But $ a_1,  \ldots, a_{t-1}, a_{t+1}, \ldots, a_k \in L$ and $o(a)=o(a_t)$, thus  $G$$=$$\langle a \rangle \times L$.
   \end{proof}

   \begin{corollary}\label{h1}
    Let $G$ be an abelian $p$-group. Then $Aut(G)$, the automorphism group of $G$, acts
    transitively on $H_t(G)$ when $H_t$ is a nonempty set.
   \end{corollary}





          \begin{lemma}\label{j1}
          Let $G$ be an  abelian $p$-group and
          $b$$\in$$G$ be a nontrivial  element of height $p^t$.
          Then
          $Aut(\mathcal{P}(N_{\overline{G}}(\overline{b})-\{\overline{x}|x
          $$\in$$ \langle b \rangle \}))$, acts
          transitively on the set of  elements of order $po(b)$
          with  equal heights in $N_G(b)-\langle b \rangle$ .

          \end{lemma}

          \begin{proof}
          Let  $K$$=$$\langle N_G(b)\rangle $,  $a^{p^t}$$=$$b$ and $o(a)$$=$$p^to(b)$. Since
            $o(a)$$=$$exp(K)$, there is a subgroup  $L$ such that $K$$=$$L$$\times$$ \langle a\rangle$. Suppose $(x,y)$$\in$$ N_G(b)$$-$$\langle b\rangle $ and $(x,y)^n$$=$$(1,b)$. Then $x^n$$=$$1$ and $y^n$$=$$b$ and
          consequently $o(x)|p^t$.
          So there are  non-isomorphic  homocyclic subgroups $L_1, \ldots,  L_m$
          such that  $ exp(L_1)$$<$$ \ldots$$ <$$exp(L_{m-1})$$<p^t$  and, $L_m$$=$$1$ or $ exp(L_m)$$=$$p^t$,  and
          $L$$=$$ L_1 \times \cdots \times L_{m-1}\times L_m^p$. Set $M$$=$$L_1 \times \cdots \times L_m$.
          Two elements $u$$=$$(x_1, \ldots, x_m, x)$ and $v$$=$$(y_1, \ldots, y_m, y)$   of order $po(b)$ in $N(b)$ have equal heights
          if and only if $o(x)$$=$$o(y)$$=$$po(b)$,
           \begin{center}{
           $\max \{o(x_1), \ldots, o(x_m), o(y_1), \ldots, o(y_m)\}|p  $}
           \end{center}
          and,
          \begin{center}{
           $\min \{i|x_i\neq 1, 1\leq i \leq m-1 \}$$=$$\min\{i|y_i \neq 1, 1\leq i \leq m-1\}$.}
           \end{center}
 we consider two cases.\\

          \textbf{Case 1}. $u, v$$\in$$ H_{exp(L_i)}$ for some $i<m$. By the proof of  Lemma \ref{h1},
          $Aut(M)$  has element $\varphi$ such that $\varphi(x_1, \ldots, x_m)$$=$$(y_1, \ldots, y_m)$ and $\varphi$ is identity on
          $L_m$. Thus $\psi$ by definition $\psi(g, a^i)$$=$$(\varphi(g),a^{ij})$ when $x^j$$=$$y$, is a group automorphism and, $\varphi(N(b)-\langle b\rangle)=N(b)-\langle b\rangle$, as required.\\

          \textbf{Case 2}. Let  $x_1$$=$$ \ldots $$=$$x_{m-1}$$=$$y_1$$=$$\ldots $$=$$y_{m-1}$$=$$1$.  Then $height(u)$$=$$height(v)$$=$$height(x)$$=$$p^{t-1}$ and there exist            $c$$\in$$ L_m \times \langle a\rangle$ such that  $o(c)$$=$$o(a)$, $u$$\in$$ \langle c \rangle$.
                          Thus $M$$=$$L\times \langle c \rangle$ and, there is $\psi $$\in$$ Aut(G)$ such that
              $\psi(a)$$=$$c$. Consequently, $\psi(x)$$=$$u$  completes the proof.

           \end{proof}
           For $x$$\in$$ G$, set $\widehat{x}$$=$$\overline{N_G(x)-\langle x \rangle}$.
          \begin{corollary}\label{al4}
          By the hypothesis of last Lemma,  
          \begin{center}
          {$Aut(\widehat{b})$$=$$
          (Aut(\widehat{(x_1,c)}) \wr S_{k_1}) \times  \ldots \times (Aut(\widehat{(x_{m-1},c)}) \wr S_{k_{m-1}})
          \times (Aut(\widehat{c})\wr S_{k_m}),$} 
          \end{center}
 where $x_i $$\in$$ H_{s_i}(\Omega_{t+1}(L))$,
           $k_i$$=$$|H_{s_i}(\Omega_{t+1}(L))|$, $k_m$$=$$p^r$,  $s_i$$=$$exp(L_i)$, $r$ is the number of direct factor of $L_m$ and, $c$ is an element of order $po(b)$ in $\langle a \rangle$.
          \end{corollary}
          \begin{proof}
          Since $\widehat{b}$ is not connected  and any component  has an unique element $\overline{u}$ of order $po(b)$, by Lemmas \ref{t2}, \ref{j1}, the result follows.
          \end{proof}

\begin{corollary} \label{abel}
 Let $G$ be abelian  $p$-group then 
                     
                     \begin{center}{
                     $Aut(\overline{\mathcal{P}(G)})$$=$$ \prod_{H_t(G)\neq \phi} (Aut(\widehat{\alpha(H_t(G)))}\wr S_{(|H_t(G)|-1)/(p-1)})$}
                     \end{center}
where $\alpha(K)$ is
an  element of prime order in $K$.
\end{corollary}
Combining Theorems \ref{autgroup},  \ref{nil} and Corollaries  \ref{al4}, \ref{abel}, automorphism group of power graph of any abelian groups  can be computed.\\



\end{document}